\documentclass[a4paper]{article}

\usepackage[margin=3cm]{geometry}
\usepackage[utf8]{inputenc}
\usepackage{amsmath}
\usepackage{amssymb}
\usepackage{amsthm}
\usepackage{mathrsfs}
\usepackage{dsfont}
\usepackage{bm}
\usepackage{graphicx}
\usepackage[pdftex,dvipsnames]{xcolor}
\usepackage{enumitem}
\usepackage[round,comma,authoryear]{natbib}
\usepackage[labelformat=simple]{caption,subcaption}
\usepackage[geometry]{ifsym}
\usepackage{soul}
\usepackage[linecolor=OliveGreen,backgroundcolor=OliveGreen!25,bordercolor=OliveGreen,textsize=small]{todonotes}
\usepackage{tabularx}
\usepackage{multirow}
\usepackage{cases}

\theoremstyle{plain}
\newtheorem{theorem}{Theorem}
\newtheorem{lemma}[theorem]{Lemma}
\newtheorem{proposition}[theorem]{Proposition}
\newtheorem{corollary}[theorem]{Corollary}

\newtheorem*{example*}{Example}

\newtheorem*{remark*}{Remark}
\newtheorem{acknowledgement*}{Acknowledgement}

\newcommand{\argmin}{\operatorname*{\arg\min}}

\newcommand{\sgn}{\operatorname{sgn}}

\newcommand{\A}{\mathcal{A}}
\newcommand{\C}{\mathcal{C}}
\newcommand{\D}{\mathcal{D}}
\newcommand{\M}{\mathcal{M}}
\newcommand{\Mhat}{\widehat{\mathcal{M}}}
\renewcommand{\O}{\mathcal{O}}
\renewcommand{\S}{\mathcal{S}}
\newcommand{\N}{\mathbb{N}}
\newcommand{\R}{\mathbb{R}}
\newcommand{\Rquer}{\overline{\R}}
\renewcommand{\P}{\mathbb{P}}
\newcommand{\ind}{\mathds{1}}
\newcommand{\eps}{\varepsilon}
\newcommand{\betaAL}{\hat\beta_{\textnormal{\tiny AL}}}
\newcommand{\betaALj}{\hat\beta_{\textnormal{\tiny AL},j}}
\newcommand{\betaLS}{\hat\beta_{\textnormal{\tiny LS}}}
\newcommand{\betaLSj}{\hat\beta_{\textnormal{\tiny LS},j}}
\newcommand{\betaLSs}{\hat\beta_{\textnormal{\tiny LS},s}}

\newcommand{\Anj}{A_{n,\beta_n,j}}
\newcommand{\Aj}{A_{\phi,j}}

\newcommand{\pto}{\overset{p}{\longrightarrow}}
\newcommand{\dto}{\overset{d}{\longrightarrow}}

\renewcommand{\emptyset}{\varnothing}

\setlist[enumerate]{ref=(\alph{enumi}),label=(\alph{enumi}),ref=(\alph{enumi})}

\begin{document}

\title{Uniform Asymptotics and Confidence Regions Based on the
Adaptive Lasso with Partially Consistent Tuning}

\author{Nicolai Amann\thanks{Department of Statistics and Operations
Research, University of Vienna, Oskar-Morgenstern Platz 1, A-1090
Vienna, nicolai.amann@univie.ac.at.}\ \
and Ulrike Schneider\thanks{Institute of Statistics and Mathematical Methods in
Economics, TU Wien, Wiedner Hauptstra{\ss}e 8/E105-2, A-1040
Vienna, ulrike.schneider@tuwien.ac.at.} \\[2ex]
University of Vienna and TU Wien}

\date{}

\maketitle

\begin{abstract}
We consider the adaptive Lasso estimator with componentwise tuning in
the framework of a low-dimensional linear regression model. In our
setting, at least one of the components is penalized at the rate of
consistent model selection and certain components may not be penalized
at all. We perform a detailed study of the consistency properties and
the asymptotic distribution which includes the effects of
componentwise tuning within a so-called moving-parameter framework.
These results enable us to explicitly provide a set $\M$ such that
every open superset acts as a confidence set with uniform asymptotic
coverage equal to 1, whereas removing an arbitrarily small open set
along the boundary yields a confidence set with uniform asymptotic
coverage equal to 0. The shape of the set $\M$ depends on the
regressor matrix as well as the deviations within the componentwise
tuning parameters. Our findings can be viewed as a broad
generalization of \cite{PoetscherSchneider09,PoetscherSchneider10} who
considered distributional properties and confidence intervals based
on components of the adaptive Lasso estimator for the case of
orthogonal regressors.
\end{abstract}

%
%
%
%

\section{Introduction} 
\label{sec:intro}

The least absolute shrinkage and selection operator or Lasso by
\cite{Tibshirani96} has received tremendous attention in the
statistics literature in the past two decades. The main attraction of
this method lies in its ability to perform model selection and
parameter estimation at very low computational cost, and the fact that
the estimator can be used in high-dimensional settings where the
number of variables $p$ exceeds the number of observations $n$ (``$p
\gg n$'').

For these reasons, the Lasso has also turned into a very popular and
powerful tool in econometrics, and similar things can be said about
the estimator's many variants, among them the adaptive Lasso estimator
of \cite{Zou06}, where the $l_1$-penalty term is randomly weighted
according to some preliminary estimator. This particular method has
been used in econometrics in the context of diffusion processes
\citep{DeGregorioIacus12}, for instrumental variables
\citep{CanerFan15}, in the framework of stationary and non-stationary
autoregressions \citep{KockCallot15,Kock16} and for autoregressive
distributed lag (ARDL) models \citep{MedeirosMendes17}, to name just a
few.

Despite the popularity of this method, there are still many open
questions on how to construct valid confidence regions in connection
with the adaptive Lasso estimator. \cite{PoetscherSchneider10}
demonstrate that the oracle property from \cite{Zou06} and
\cite{HuangEtAl08b} cannot be used to conduct valid inference and that
resampling techniques also fail. They give confidence intervals with
exact coverage in finite samples as well as an extensive asymptotic
study in the framework of orthogonal regressors. However, settings
more general than the orthogonal case have not been considered yet.

\smallskip

In this paper, we consider an arbitrary low-dimensional linear
regression model ($``p \leq n$'') where the regressor matrix exhibits
full column rank. We allow for the adaptive Lasso estimator to be
tuned componentwise with some tuning parameters possibly being equal
to zero, so that not all coordinates have to be penalized. Due to this
componentwise structure, three possible asymptotic regimes arise: the
one where each zero component is identified as such with asymptotic
probability less than one, usually termed \emph{conservative model
selection}, the one where each zero component is revealed as zero with
asymptotic probability equal to one, usually referred to as
\emph{consistent model selection}, as well as the mixed case where
some components are tuned conservatively and some are tuned
consistently. The framework we consider encompasses the latter two
regimes.

\smallskip

The main challenge for inference in connection with the adaptive Lasso
and related estimators lies in the fact that the finite-sample
distribution depends on the unknown parameter in a complicated manner,
and that this dependence persists in large samples. Consequently, the
coverage probability of a confidence region varies over the parameter
space, and in order to conduct valid inference, one needs to guard
against the lowest possible coverage and consider the minimal one.
This is done so in the present paper.

\smallskip

Since explicit expressions for the finite-sample distribution and
therefore also the coverage probabilities of confidence regions are
unknown when the regressors are not orthogonal, our study is set in an
asymptotic framework. We determine the appropriate uniform rate of
convergence and derive the asymptotic distribution of an appropriately
scaled estimator that has been centered at the true parameter. While
the limit distribution is still only implicitly defined through a
minimization problem, the key observation and finding is that one may
explicitly characterize the set of minimizers once the union over all
true parameters is taken. This is done by heavily exploiting the
structure of the corresponding optimization problem and leads to a
compact set $\M$ that is determined by the asymptotic Gram matrix as
well as the asymptotic deviations between the componentwise tuning
parameters and the maximal one. Subsequently, this result can be used
to show how the set $\M$ acts as a benchmark for confidence regions
since, very loosely put, any larger set will necessarily have
asymptotic coverage equal to one, and any smaller set will exhibit
zero uniform coverage in the limit. We will formalize this statement,
sharpen it for a wide class of tuning regimes and demonstrate the
intrinsic problem that arises for inference in connection with this
estimation method.

\smallskip

In this article, we show that the one-dimensional case from
\cite{PoetscherSchneider10} can indeed be generalized to arbitrary
low-dimensional models. Our investigations reveal the geometry of
confidence regions based on the adaptive Lasso which cannot be seen in
the intervals of the one-dimensional setting. Our study also
encompasses the effects of varying tuning schemes over different
components of the parameter which may result in a loss of dimension in
the confidence set.

\smallskip

The paper is organized as follows. We introduce the model and the
assumptions as well as the estimator in Section~\ref{sec:setting}. In
Section~\ref{sec:AL_LS}, we study the relationship of the adaptive
Lasso to the least-squares estimator. The consistency properties with
respect to parameter estimation, rates of convergence, and model
selection are derived in Section~\ref{sec:consistency}.
Section~\ref{sec:asymp_dist} looks at the asymptotic distribution of
the estimator and deduces that it is always contained in a compact
set, independently of the unknown parameter. These results are used to
construct and discuss the confidence regions in
Section~\ref{sec:conf_sets}, where their shape is also illustrated. We
summarize in Section~\ref{sec:summary} and relegate all proofs to
Appendix~\ref{sec:proofs} for readability.

\section{Setting and Notation} 
\label{sec:setting}

We consider the linear regression model
$$
y = X\beta + \eps,
$$
where $y \in \R^n$ is the response vector, $X \in \R^{n \times p}$ the
non-stochastic regressor matrix assumed to have full column rank,
$\beta \in \R^p$ the unknown parameter vector and $\eps \in \R^n$ the
unobserved stochastic error term consisting of independent and
identically distributed components with mean zero and finite second
moments, defined on some probability space
$(\Omega,\mathcal{F},\mathbb{P})$. To define the adaptive Lasso
estimator, first introduced by \cite{Zou06}, let
$$
L_n(b) = \|y - Xb\|^2 + 2 \sum_{j=1}^p \lambda_j \frac{|b_j|}{|\betaLSj|},
$$
where $\|.\|$ is the Euclidean norm, $\lambda_j$ are non-negative
tuning parameters, and $\betaLS = (X'X)^{-1}X'y$ is the ordinary
least-squares (LS) estimator. We assume the event $\{\betaLSj = 0\}$
to have zero probability for all $j=1,\dots,p$ and do not consider
this event occurring in the subsequent analysis. The adaptive Lasso
estimator we employ is given by
$$
\betaAL = \argmin_{b \in \R^p} L_n(b),
$$
which always exists and is uniquely defined in our setting. Note that,
in contrast to \cite{Zou06}, we allow for \emph{componentwise partial
tuning} where the tuning parameter may vary over coordinates and may
be equal to zero, so that not all components need to be penalized.
This is unlike the typical case of \emph{uniform tuning} with a single
positive tuning parameter. We also look at the leading case of
$\omega_j = 1/|\betaLSj|^\gamma$ with $\gamma = 1$, in the notation of
\cite{Zou06}. For all asymptotic considerations, we will assume that
$X'X/n$ converges to a positive definite matrix $C \in \R^{p \times
p}$ as $n \to \infty$.

\smallskip

We define the true active set $\A$ to be $\A = \{j: \beta_j \neq 0\}$.
The quantity $\lambda^*$ is given by the largest tuning parameter,
$\lambda^* = \max_{1 \leq j \leq p} \lambda_j$. We use $\Rquer$ for
the extended real line. Finally, the symbol $\dto$ stands for
convergence in distribution. For the sake of readability, we suppress
the dependence of the following quantities on $n$ in the notation:
$y$, $X$, $\eps$, $\betaAL$, $\betaLS$, $\lambda_j$ and $\lambda^*$.

\section{Relationship to LS estimator} \label{sec:AL_LS}

The following finite-sample relationship between the adaptive Lasso
and the LS estimator is essential for proving the results in the
subsequent section and will also give some insights for understanding
the idea behind the results on the shape of the confidence regions in
Sections~\ref{sec:asymp_dist} and \ref{sec:conf_sets}. The lemma shows
that the difference between the adaptive Lasso and the LS estimator is
always contained in a bounded and closed set that depends on the
regressor matrix as well as on the tuning parameters. Note that the
statements in Lemma~\ref{lem:AL_LS} and Corollary~\ref{cor:AL_LS} hold
for all $\omega \in \Omega$, i.e., ``surely''.
\begin{lemma}[Relationship to LS estimator] \label{lem:AL_LS}
$$
\betaAL - \betaLS \in \{z \in \R^p: (X'Xz)_j = 0 \text{ for } \lambda_j = 0, z_j(X'X z)_j \leq
\lambda_j \text{ for } \lambda_j > 0, j = 1,\dots,p\} \;
$$
for all $\omega \in \Omega$.
\end{lemma}

Lemma~\ref{lem:AL_LS} can be used to determine under which tuning
regime the adaptive Lasso is asymptotically behaving the same as the
LS estimator, as is stated in the following corollary.

\begin{corollary}[Equivalence to LS estimator] \label{cor:AL_LS}
If $\lambda^* \to 0$, $\betaAL$ and $\betaLS$ are asymptotically equivalent in the sense that
$$
\sqrt{n}(\betaAL - \betaLS) \to 0 \;\;\; \text{ as } n \to \infty \;
\text{ for all } \omega \in \Omega.
$$
\end{corollary}

Corollary~\ref{cor:AL_LS} shows that in case $\lambda^* \to 0$, the
adaptive Lasso estimator is asymptotically equivalent to the LS
estimator, so that this case becomes a trivial one. How the estimator
behaves in terms of parameter estimation and model selection for
different asymptotic tuning regimes is treated in the next section.

\section{Consistency in parameter estimation and model selection} \label{sec:consistency}

We start our investigation by deriving the pointwise convergence rate
of the estimator.

\begin{proposition}[Pointwise convergence rate]  \label{prop:rate_point}
Let $a_n = \min(\sqrt{n},n/\lambda^*)$. Then the adaptive Lasso
estimator is pointwise $a_n$-consistent for $\beta$ in the sense that
for every $\delta > 0$, there exists a real number $M_{\beta,\delta}$
such that
$$
\sup_{n \in \N}\ \P_\beta \left(a_n \|\betaAL-\beta\| >
M_{\beta,\delta} \right) \leq \delta.
$$
\end{proposition}

\smallskip

The fact that the pointwise convergence rate is given by $n^{1/2}$
only if $\lambda^*/n^{1/2}$ does not diverge has implicitly been noted
in \cite{Zou06}'s oracle property in Theorem~2 in that reference,
reflected in the assumption of $\lambda^*/n^{1/2} \to 0$\footnote{Note
that $\lambda_n$ in that reference corresponds to $2\lambda^*$ in our
notation, assuming uniform tuning over all components.}. In the
one-dimensional case, it can be learned from Theorem~5 Part~2 in
\cite{PoetscherSchneider09} that the sequence $n^{1/2}(\betaAL -
\beta)$ is not stochastically bounded if $\lambda^*/n^{1/2}$
diverges\footnote{To make the connection from that reference to our
notation, note that $p=1$ there and set $\theta_n = \beta$ and
$n\mu_n^2 = \lambda^*$.}. However, neither of these references
determine the slower rate of $n/\lambda^*$ explicitly when it applies.

The uniform convergence rate is presented in the next proposition.

\begin{proposition}[Uniform convergence rate] \label{prop:rate_unif}
Let $b_n = \min(\sqrt{n},\sqrt{n/\lambda^*})$. Then the adaptive Lasso
estimator is uniform $b_n$-consistent for $\beta$ in the sense that
for every $\delta > 0$, there exists a real number $M_\delta$ such
that
$$
\sup_{n \in \N}\ \sup_{\beta \in \R^p}\ \P_\beta \left(b_n
\|\betaAL-\beta\| > M_\delta \right) \leq \delta.
$$
\end{proposition}

Proposition~\ref{prop:rate_unif} shows that the uniform convergence
rate is slower than $n^{1/2}$ if $\lambda^* \to \infty$. The fact that
the uniform rate may differ from the pointwise one has been noted in
\cite{PoetscherSchneider09}. Unless the estimator is inconsistent in
parameter estimation, the uniform convergence from
Proposition~\ref{prop:rate_unif} is slower than the pointwise one and
can, indeed, not be improved upon. The latter statement is
substantiated by Theorem~\ref{thm:asymp_dist} in
Section~\ref{sec:asymp_dist}, which shows that the limit of
$b_n(\betaAL - \beta_n)$ is non-zero for certain sequences $\beta_n$.

\begin{theorem}[Consistency in parameter estimation] \label{thm:consist_param}
The following statements are equivalent.

\begin{enumerate} 

\item \label{item:ptw_consist} $\betaAL$ is pointwise consistent for $\beta$.

\item \label{item:unif_consist} $\betaAL$ is uniformly consistent for $\beta$.

\item \label{item:cond_lambda} $\lambda^*/n \to 0$ as $n \to \infty$.

\item \label{item:correct_ms} $\lim\limits_{n \to \infty} \P_\beta(\betaALj =
0) = 0$ whenever $j \in \A$.

\end{enumerate}

\end{theorem}

Condition \ref{item:correct_ms} in Theorem~\ref{thm:consist_param}
states that the adaptive Lasso never chooses underparametrized models
with asymptotic probability equal to 1. It underlines the fact that
$\lambda^*/n \to 0$ is a basic condition that we will assume in all
subsequent statements.

\begin{theorem}[Consistency in model selection] \label{thm:consist_model}
Suppose that $\lambda^*/n \to 0$ as $n \to \infty$. If $\lambda_j \to
\infty$ as well as $\sqrt{n}\lambda_j/\lambda^* \to \infty$ as $n \to
\infty$ for all $j = 1,\dots,p$, then the adaptive Lasso estimator
performs consistent model selection in the sense that
$$
\lim_{n \to \infty}\P_\beta(\betaALj \neq 0 \iff j \in \A) = 1 \text{ as } n \to
\infty.
$$
\end{theorem}

\begin{remark*}
Inspecting the proof of Theorem~\ref{thm:consist_model} shows that in
fact a more refined statement than Theorem~\ref{thm:consist_model}
holds. Assume that $\lambda^*/n \to 0$. We then have that
$\P_\beta(\betaALj = 0) \to 0$ whenever $j \in \A$ and
$$
\lambda_j \to \infty \text{ and } \frac{\sqrt{n}\lambda_j}{\lambda^*}
\to \infty \implies \lim_{n \to \infty} \P_\beta(\betaALj = 0) = 1
\text{ for } j \notin \A \implies \lambda_j \to \infty.
$$
This statement is in particular interesting for the case of partial
tuning where some $\lambda_j$ are set to zero and the corresponding
components are not penalized, revealing that the other components can
still be tuned consistently in this case.
\end{remark*}

\section{Asymptotic distribution} \label{sec:asymp_dist}

In this section we investigate the asymptotic distribution. We
perform our analysis for the case when $\lambda^* \to \infty$ which,
by Theorem~\ref{thm:consist_model}, encompasses the tuning regime of
consistent model selection and often is the regime of choice in
applications. If the estimator is tuned uniformly over all components,
the condition $\lambda^* \to \infty$ is in fact equivalent to
consistent tuning, given the basic condition of $\lambda^*/n \to 0$.

The requirement $\lambda^* \to \infty$ also corresponds to the case
where the convergence rate of the adaptive Lasso estimator is given by
$(n/\lambda^*)^{1/2}$ rather than $n^{1/2}$, as can be seen from
Proposition~\ref{prop:rate_unif}. \cite{PoetscherSchneider09,
PoetscherSchneider10} demonstrate that in order to get a
representative and full picture of the behavior of the estimator from
asymptotic considerations, one needs to consider a moving-parameter
framework where the unknown parameter $\beta = \beta_n$ is allowed to
depend on sample size. For these reasons, we study the asymptotic
distribution of $(n/\lambda^*)^{1/2}(\betaAL - \beta_n)$, which is
done in the following.

Throughout Section~\ref{sec:asymp_dist} and
Section~\ref{sec:conf_sets}, let $\lambda^0 \in [0,1]^p$ and $\psi \in
[0,\infty]^p$ be defined by
\begin{align*}
\frac{\lambda_j}{\lambda^*} & \to \lambda^0_j \in [0,1] \text{ and }\\[1ex]
\frac{\sqrt{\lambda^*}}{\lambda_j} & \to \psi_j \in [0,\infty], 
\end{align*}
measuring the two different deviations between each tuning parameter
to the maximal one. Note that we have $\lambda^0 = (1,\dots,1)'$ and
$\psi = 0$ for uniform tuning, and that not penalizing the $j$-th
parameter leads to $\psi_j = \infty$ and $\lambda^0_j = 0$. Note that
assuming the existence of these limits does not pose a restriction, as
we could always perform our analyses on convergent subsequences and
characterize the limiting behavior for all accumulation points.

\begin{theorem}[Asymptotic distribution] \label{thm:asymp_dist}
Assume that $\lambda^*/n \to 0$ and $\lambda^* \to \infty$. Moreover,
define $\phi \in \Rquer^p$ by
$\sqrt{n}\beta_{n,j}\sqrt{\lambda^*}/\lambda_j \to \phi_j$ for $j =
1,\dots,p$. Then
$$
\sqrt{\frac{n}{\lambda^*}}(\betaAL - \beta_n) \dto \argmin_{u \in \R^p}
V_\phi(u),
$$
where 
$$
V_\phi(u) = u'Cu + \sum_{j=1}^p \begin{cases}
0 & u_j=0 \text{ or } |\phi_j| = \infty \text{ or } \psi_j = \infty \\ 
\infty & u_j \neq 0 \text{ and } \phi_j = \psi_j = 0 \\ 
2 \frac{|u_j + \lambda^0_j\phi_j| - |\lambda^0_j\phi_j|}{|\phi_j + \psi_jZ_j|} & \text{else},
\end{cases}
$$
with $Z \sim N(0,\sigma^2 C^{-1})$, where $X'X/n \to C$, positive
definite.
\end{theorem}

There are a few things worth mentioning about
Theorem~\ref{thm:asymp_dist}. First of all, in contrast to the
one-dimensional case, the asymptotic limit of the appropriately scaled
and centered estimator may still be random. However, this can only
occur if $\psi_j$ is non-zero and finite for some component $j$,
meaning that the maximal tuning parameter diverges faster (in some
sense) than the tuning parameter for the $j$-th component, but not too
much faster. When no randomness occurs in the limit, the rate of the
stochastic component of the estimator is obviously smaller by an order
of magnitude compared to the bias component. In particular, this will
always be the case for uniform tuning when $\psi = 0$.

As is expected, the proof of Theorem~\ref{thm:asymp_dist} will be
carried out by looking at the corresponding asymptotic minimization
problem of the quantity of interest, which can shown to be the
minimization of $V_\phi$. However, since this limiting function is not
finite on an open subset of $\R^p$, the reasoning of why the
appropriate minimizers converge in distribution to the minimizer of
$V_\phi$ is not as straightforward as might be anticipated.

The assumption of $n^{1/2}\beta_n\lambda^{*1/2}/\lambda_j$ converging
in $\Rquer^p$ in the above theorem is not restrictive in the sense
that otherwise, we simply revert to converging subsequences and
characterize the limiting behavior for all accumulation points, which
will prove to be all we need for Proposition~\ref{prop:min_set} and
the confidence regions in Section~\ref{sec:conf_sets}.

While we cannot explicitly minimize $V_\phi$ for a fixed $\phi \in
\R^p$ other than in trivial cases, surprisingly, we can still
explicitly deduce the set of all minimizers of $V_\phi$ over all $\phi
\in \R^p$, which yields the same set regardless of the realization of
$Z$ in $V_\phi$. This is done in the following proposition.

\begin{proposition}[Set of minimizers] \label{prop:min_set}
Define
$$
\M =\M(\lambda^0,\psi) = 
\left\{m \in \R^p: (Cm)_j = 0 \text{ if } \psi_j = \infty, \, m_j(Cm)_j \leq \lambda^0_j 
\text{ if } \psi_j < \infty\right\}.
$$
Then for any $\omega \in \Omega$ we have
$$
\M = \bigcup_{\phi \in \Rquer^p} \argmin_{u \in \R^p} V_\phi(u)(\omega).
$$
\end{proposition}

So, while the limit of $(n/\lambda^*)^{1/2}(\betaAL - \beta_n)$ will,
in general, be random, the set $\M$ is not. In fact,
Proposition~\ref{prop:min_set} shows that for any $\omega$, the union
of limits over all possible sequences of unknown parameters is always
given by the same compact set $\M$. This observation is central for
the construction of confidence regions in the following section. It
also shows that while in general, a stochastic component will survive
in the limit, it is always restricted to have bounded support that
depends on the regressor matrix and the tuning parameter through the
matrix $C$ and the quantities $\psi$ and $\lambda^0$. Interestingly,
$\M$ only depends on $\psi$ for the components where $\psi_j =
\infty$, in which case the set $\M$ loses a dimension. This can be
seen as a result of the $j$-th component being penalized much less
than the maximal one so that the scaling factor used in
Theorem~\ref{thm:asymp_dist} is not large enough for this component to
survive in the limit. Note that in case of uniform tuning where $\psi
= 0$ and $\lambda^0 = (1,\dots,1)'$, $\M$ does not depend on the
sequence of tuning parameters at all. Also, we have $\M = [-1,1]$ for
$p=1$ and $C=1$, a fact that has been shown in
\cite{PoetscherSchneider09} and used in \cite{PoetscherSchneider10}.

A simple ``quick-and-dirty'' way to motivate the result in
Proposition~\ref{prop:min_set} is to rewrite
$$
\sqrt{\frac{n}{\lambda^*}}(\betaAL - \beta_n) = 
\sqrt{\frac{n}{\lambda^*}}(\betaAL - \betaLS) + \sqrt{\frac{n}{\lambda^*}}(\betaLS - \beta_n) 
$$
and observe that the second term on the right-hand side is $o_p(1)$
whereas the first term is always contained in the set 
$$
\left\{z \in \R^p : z_j(\frac{X'X}{n}z)_j \leq
\frac{\lambda_j}{\lambda^*} \text{ for } j=1,\dots,p\right\}
$$
by Lemma~\ref{lem:AL_LS}, which contains the set $\M$ in the limit.
Theorem~\ref{thm:asymp_dist} and Proposition~\ref{prop:min_set} can
therefore be viewed as the theory that makes this observation precise
by sharpening the set and showing that it only contains the limits.
This can then be used for constructing confidence regions, which is
done in the following section.

\section{Confidence regions -- coverage and shape} \label{sec:conf_sets}

The insights from Theorem~\ref{thm:asymp_dist} and
Proposition~\ref{prop:min_set} can now be used for deriving the
following theorem on confidence regions.

\begin{theorem}[Confidence regions] \label{thm:conf_sets}
Let $\lambda^*/n \to 0$ and $\lambda^* \to \infty$. Then every open
superset $\O$ of $\M$ satisfies
$$
\lim_{n \to \infty} \inf_{\beta \in \R^p} 
P_\beta(\beta \in \betaAL - \sqrt{\frac{\lambda^*}{n}}\O) = 1.
$$
For $d > 0$, define $\M_d = \M(d\lambda^0,\psi)$. We then have that
$$
\lim_{n \to \infty} \inf_{\beta \in \R^p} 
P_\beta(\beta \in \betaAL - \sqrt{\frac{\lambda^*}{n}}\M_d) = 0
$$
for any $0 < d < 1$.
\end{theorem}

\begin{remark*}
The statements in Theorem~\ref{thm:conf_sets} can be strengthened in
the following way. Let $\lambda^*/n \to 0$ and $\lambda^* \to \infty$.

\begin{enumerate}[label=(\alph{enumi}),ref=(\alph{enumi})]

\item \label{item:conf_ref_super} If $\lambda^0 \in (0,1]^p$, then for
any $d > 1$ we have
$$
\lim_{n \to \infty} \inf_{\beta \in \R^p} 
P_\beta(\beta \in \betaAL - \sqrt{\frac{\lambda^*}{n}}\M_d) = 1.
$$

\item \label{item:conf_ref_sub} If $\psi \in \{0,\infty\}^p$, then any
closed and proper subset $\C$ of $\M$ fulfills
$$
\lim_{n \to \infty} \inf_{\beta \in \R^p} 
P_\beta(\beta \in \betaAL - \sqrt{\frac{\lambda^*}{n}}\C) = 0.
$$
\end{enumerate}
Note that for uniform tuning, both refinements hold since $\psi=0$ and
$\lambda^0 = (1,\dots,1)'$.

Part \ref{item:conf_ref_super} holds since under the given
assumptions, $\M_d$ has non-empty interior and therefore contains an
open superset of $\M$. Part \ref{item:conf_ref_sub} hinges on the fact
that the limits in Theorem~\ref{thm:asymp_dist} are always non-random
under the given assumptions.
\end{remark*}

Casually put, Theorem~\ref{thm:conf_sets} and the subsequent remark
show the following. The set $\M = \M_1$ acts as a benchmark for
confidence sets in the sense that if we take a ``slightly larger''
set, multiplied with the appropriate factor and centered at the
adaptive Lasso estimator, we get a confidence region with minimal
asymptotic coverage probability equal to 1. If, however, we base the
region on a ``slightly smaller'' set than $\M$, we end up with a
confidence set of asymptotic minimal coverage 0. Nothing can be
revealed from the above when using $\M$ itself. We get into a deeper
discussion in the following.

We focus on the case where $\lambda^0 \in (0,1]^p$, i.e., the case
where all components of $\lambda^0$ are non-zero (implying $\psi =
0$). This means that all components are penalized at the same rate,
which is obviously fulfilled for uniform tuning. In this case, the
asymptotic distribution is mere point-mass with no stochastic part
surviving in the limit, as can be seen from
Theorem~\ref{thm:asymp_dist}. The reason for this is the fact that
when controlling for the bias of the estimator (by scaling with the
reciprocal of the uniform convergence rate), the stochastic part
vanishes asymptotically. In other words, the appropriate scaling
factor is simply not large enough to keep the random component alive
in the limit, illustrating that the bias is of larger order than the
stochastic component when viewed under a uniform lens\footnote{Note
that Proposition~\ref{prop:min_set} shows that in \emph{all} settings
where at least one component is tuned consistently, even if a
stochastic component survives in the limit, it always has bounded
support contained in $\M$, leaving very limited possibilities for the
construction of confidence regions based on the asymptotic
distribution.} -- a fact that is generally inherent to penalized
estimators.

Given the above considerations, one might ask what happens when the
confidence region is based on $\M_{d_n}$ where $d_n$ may vary? The
following theorem addresses this question by giving upper and lower
bounds for the corresponding coverage probabilities. To state the
theorem, we define the finite sample version of $\M_d$ as
$$
\Mhat_d = \left\{ m \in \R^p: (X'X m)_j = 0 \text{ if }
\lambda_j = 0, m_j \left(\frac{X'X}{n} m \right)_j \leq
\frac{\lambda_j}{\lambda^*} d \text{ if } \lambda_j > 0\right\},
$$
which differs from $\M_d$ only in that $C$ and $\lambda^0$ are
replaced by their finite-sample equivalents $X'X/n$ and
$\lambda/\lambda^*$, respectively, so that $\Mhat_d$ converges to
$\M_d$ (in the Hausdorff metric). We now provide lower and upper
bounds depending on if and how $d_n$ converges to $1$ in relation to
$\lambda^*$.

\begin{theorem} \label{thm:conf_bounds}
Assume that $\lambda^0 \in (0,1]^p$ and let $\nu = \lim_{n \to \infty}
\sqrt{\lambda^*} (d_n-1) \in \Rquer$. We then have
$$
\limsup_{n \to \infty} \inf_{\beta \in \R^p}
\P_\beta \left(\beta \in \betaAL - \sqrt{\frac{\lambda^*}{n}}\Mhat_{d_n} \right)
\leq \min_{1 \leq j \leq p} 
\Phi\left(\frac{\nu\sqrt{\lambda^0_j}}{\sigma\sqrt{3+(C^{-1})_{jj}C_{jj}}} \right),
$$
and for $\nu > 0$
$$
\liminf_{n \to \infty}\inf_{\beta \in \R^p}
\P_\beta \left( \beta \in \betaAL - \sqrt{\frac{\lambda^*}{n}}\Mhat_{d_n} \right) \geq 
\min_{1 \leq j \leq p}
F_{\chi^2_p} \left(\frac{(\lambda^0_j\nu)^2}{4 \kappa_{C} l_0\sigma^2} \right),
$$
where $l_0 = \sum_{j=1}^p \lambda^0_j$, and $\Phi$ and $F_{\chi^2_p}$
denote the cdf of a standard normal and a chi-squared distribution
with $p$ degrees of freedom, respectively. The symbol $\kappa_{C}$
stands for the condition number of $C$ with respect to the spectral
norm, i.e., the ratio of the largest and the smallest eigenvalue.
\end{theorem}

\begin{remark*} \begin{enumerate} \item Theorem~\ref{thm:conf_bounds} can be shown
to still hold true when $\Mhat_{d_n}$ is replaced by its counterpart
$\M_{d_n}$, with a slight adaptation of the constant $\nu$ involving
the convergence rate of $X'X/n$ to $C$ and $\lambda/\lambda^*$ to
$\lambda^0$.

\item If $d_n = 1$ for all $n$, implying that the confidence region is
based on $\Mhat_1$, the above theorem provides $0$ as lower and $1/2$
as upper bound. The lower bound can, in fact, be shown to be
\emph{strict}, implying that using $\Mhat_1$ will always yield a
\emph{positive} asymptotic coverage (bounded by $1/2$) when all
components of $\lambda^0$ are non-zero.

\item \cite{PoetscherSchneider10} prove that in the one-dimensional and
Gaussian case, the upper bound of Theorem~\ref{thm:conf_bounds} is
sharp: The interval $[\betaAL - (\lambda^*/n)^{1/2}d_n, \betaAL +
(\lambda^*/n)^{1/2}d_n]$ possesses asymptotic infimal coverage
probability of $\Phi(\nu/(2\sigma))$, which is precisely the upper
bound in the above theorem. 

\item Lemma~\ref{lem:finIneq}, on which the proof of the second
statement in the above theorem is based, reveals that for any $d > 1$,
the convergence rate of the coverage probability of $\Mhat_d$
(converging to 1) is \emph{at least} $1/\lambda^*$.

\end{enumerate}

\end{remark*}

Theorem~\ref{thm:conf_bounds} furthermore allows to illustrate the
following. Assume that the confidence region $\Mhat_{d_n}$ has
asymptotic coverage strictly between 0 and 1 (implying that $d_n \to
1$). Then this region will asymptotically not differ in volume from
sets that exhibit asymptotic coverage of probability 1. In fact, it
can be shown that there exists a sequence $\tilde d_n$ such that
$\Mhat_{\tilde d_n}$ has asymptotic coverage 1, satisfying
$$
(\lambda^*)^q \left(\frac{\mu_p(\Mhat_{\tilde d_n})
}{\mu_p(\Mhat_{d_n})} - 1 \right) \longrightarrow 0 \;\;\; \text{ for
all } q < \frac{1}{2},
$$
where $\mu_p$ denotes $p$-dimensional Lebesgue measure. This states
that the ratio of volumes will tend to $1$, even faster than rate
$(\lambda^*)^q$ for any $q < 1/2$. It demonstrates a peculiar
nature inherent to the estimation method, differing strongly from the
standard approach through the LS estimator.

\medskip

One might wonder now how this type of confidence region does indeed
compare to the confidence ellipse based on the LS estimator. Note that
the regions will be multiplied by a different factor and centered at a
different estimator. In general, the following observation can be
made. For $0 < \alpha < 1$, let $E_\alpha = \{z \in \R^p : z'Cz \leq
k_\alpha\}$ with $k_\alpha > 0$ be such that $\betaLS -
n^{-1/2}E_\alpha$ is an asymptotic $(1\!-\!\alpha)$-confidence region
for $\beta$. If we contrast this with $\betaAL -
(\frac{\lambda^*}{n})^{1/2}\M$, we see that since both $E_\alpha$ and
$\M$ have positive, finite volume and since $\lambda^* \to \infty$,
the regions based on the adaptive Lasso are always larger by an order
of magnitude. This phenomenon is a special case of what has been found
for any consistently tuned model selection estimator in
\cite{Poetscher09}.

\medskip

Finally, we illustrate the shape of $\M$. We start with $p=2$ and the
matrix
$$
C = \begin{bmatrix} 1 & -0.7 \\ -0.7 & 1 \end{bmatrix}.
$$
We consider the case of uniform tuning, so that $\lambda^0 = (1,1)'$
and $\psi = (0,0)'$ and show the resulting set $\M$ in
Figure~\ref{fig:conf_2d}. The color indicates the value of
$\max_{j=1,2} m_j(Cm)_j$ at the specific point $m$ inside the set. The
higher the absolute value of the correlation of the covariates, the
flatter and more stretched the confidence set becomes. As one may
expect intuitively, in case of negative correlation, the confidence
set covers more of the area where the signs of the covariates are
equal, as can be seen in Figure~\ref{fig:conf_2d}. A positive
correlation causes the opposite behavior. Note that the corners of the
set $\M$ touch the boundary of the ellipse $E_\alpha$ for a certain
value of $k_\alpha$.

\begin{figure}[htp]
\centering
\includegraphics[width=0.9\textwidth,height=0.7\textwidth]{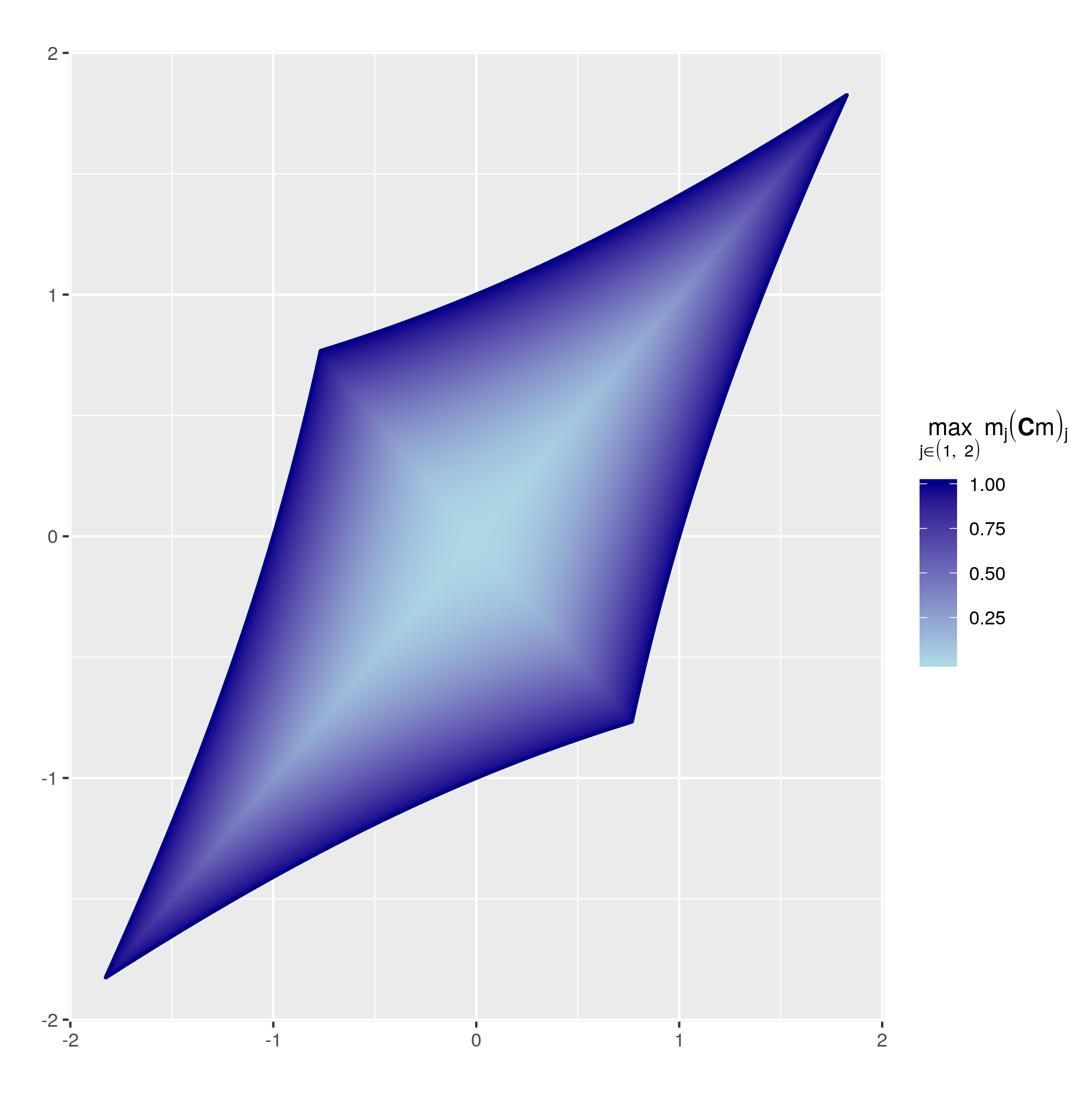}
\caption{\label{fig:conf_2d} An example for the set $\M$ with uniform
tuning in $p=2$ dimensions.}
\end{figure}

For the case of $p=3$, we again start with an example with uniform
tuning so that $\lambda^0 = (1,1,1)'$ and $\psi = (0,0,0)'$ and
consider the matrix
$$
C = \begin{bmatrix} 1 & -0.3 & 0.7 \\ -0.3 & 1 & 0.2 \\ 0.7 & 0.2 & 1 \end{bmatrix}.
$$
The resulting set $\M$ is depicted in Figure~\ref{fig:conf_3d}. To
give a better impression of the shape, the set is colored depending on
the value of the third coordinate. Here, the high correlation between
the first and third covariate stretches the set in the direction where
the signs of the covariates differ. Figure~\ref{subfig:conf_proj}
shows the projections of the three-dimensional set of
Figure~\ref{subfig:conf_3d} onto three planes where one component is
held fixed at a time. The projection onto the plane where the second
component is held constant clearly shows the behavior explained above.
On the other hand, the other two projections emphasize that for
covariates with a lower correlation in absolute value, the confidence
set is less distorted.

\begin{figure}[htp]
\centering
\begin{subfigure}{\textwidth}
\centering
\includegraphics[width=0.9\textwidth,height=0.7\textwidth]{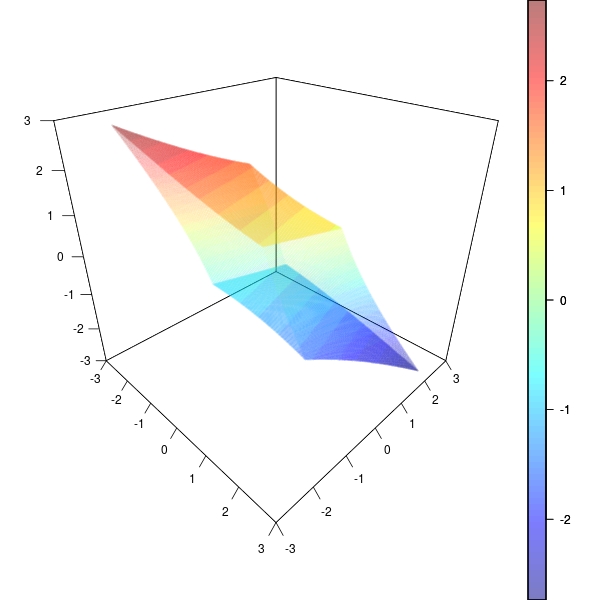}
\caption{\label{subfig:conf_3d}}
\end{subfigure}
\begin{subfigure}{\textwidth}
\centering
\includegraphics[width=0.9\textwidth,height=0.7\textwidth]{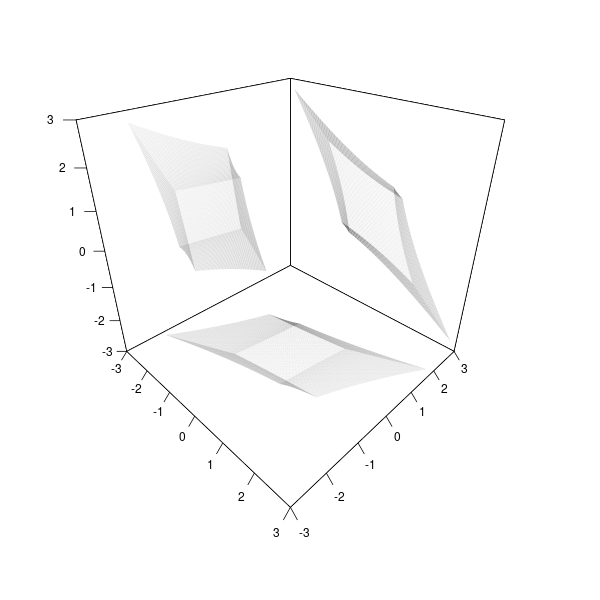}
\caption{\label{subfig:conf_proj}}
\end{subfigure}
\caption{\label{fig:conf_3d} An example for the set $\M$ with uniform
tuning and $p=3$ dimensions. The three-dimensional set is depicted in
(\subref{subfig:conf_3d}) whereas its two-dimensional projections are
shown in (\subref{subfig:conf_proj}).}
\end{figure}

Finally, Figure~\ref{fig:conf_partial} illustrates the partially tuned
case with the same matrix $C$. The first component is not penalized
whereas the remaining ones are tuned uniformly. This implies that
$\lambda^0 = (0,1,1)'$ and $\psi = (\infty,0,0)'$. Due to the
condition $(Cm)_1 = 0$ for all $m \in \M$, the resulting set is an
intersection of a plane with the set in Figure~\ref{subfig:conf_3d}.
The fact that the confidence set is only two-dimensional might appear
odd and is due to the fact that the unpenalized component exhibits a
faster convergence rate so that the factor $(\lambda^*/n)^{1/2}$ with
which $\M$ is multiplied is not large enough for this component to
survive in the limit.

\begin{figure}[htp]
\centering
\includegraphics[width=0.9\textwidth,height=0.7\textwidth]{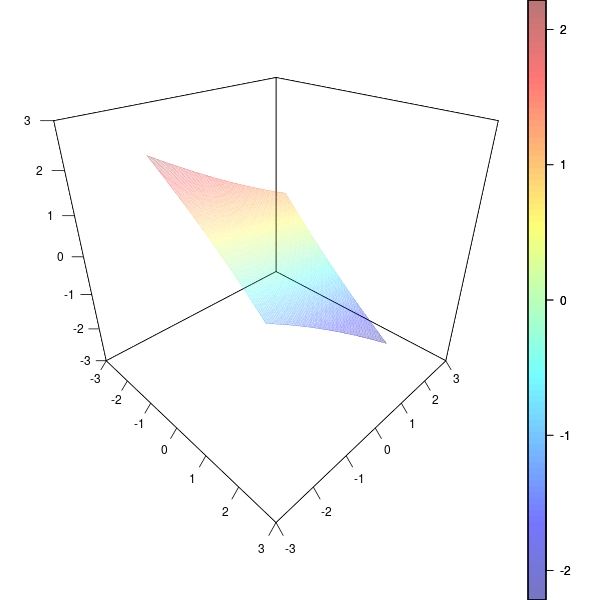}
\caption{\label{fig:conf_partial} An example of the set $\M$ with
partial tuning and $p=3$ dimensions. The first component is not
penalized resulting in the set being part of a two-dimensional
subspace.}
\end{figure}

\section{Summary and conclusions} \label{sec:summary}

We give a detailed study of the asymptotic behavior of the adaptive
Lasso estimator with partially consistent and partial tuning in a
low-dimensional linear regression model in terms of consistency and
distributional properties. We do so within a framework that takes into
account the non-uniform behavior of the estimator, non-trivially
generalizing results from \cite{PoetscherSchneider09} that were
derived for the case of orthogonal regressors. We also demonstrate and
formalize what these distributional results imply for valid confidence
regions, namely that there exists a ``benchmark'' set $\M$, such that
open supersets have asymptotic coverage equal to 1, whereas ``slightly
smaller'' sets exhibit 0 uniform coverage in the limit. The reason for
this phenomenon lies in the different rates of the bias component and
the stochastic component of the estimator. A similar effect has been
observed before for the one-dimensional case in
\cite{PoetscherSchneider10}. We illustrate the shape of $\M$ and
demonstrate the effect of componentwise tuning at different rates, as
well as the implications of partial tuning on the confidence set.

\newpage

\appendix
\section{Appendix -- Proofs} \label{sec:proofs}

We introduce the following additional notation for the proofs. The
symbol $e_j$ denotes the $j$-th unit vector in $\R^p$ and the sign
function is given by $\sgn(x) = \ind_{\{x > 0\}} - \ind_{\{x < 0\}}$
for $x \in \R$. For a function $g:\R^p \to \R$, the one-sided
directional derivative of $g$ at $u$ in the direction of $r \in \R^p$
is denoted by $\D_r g(u)$, given by
$$
\D_r g(u) = \lim_{h \searrow 0} \frac{g(u + hr) - g(u)}{h}.
$$
For a vector $u \in \R^p$ and an index set $I \subseteq
\{1,\dots,p\}$, $u_I \in \R^{|I|}$ contains only the components of $u$
corresponding to indices in $I$. Finally, $\pto$ denotes convergence
in probability.

\subsection{Proofs for Section~\ref{sec:AL_LS}}

\begin{proof}[Proof of Lemma~\ref{lem:AL_LS}]
Consider the function $G_n: \R^p \to \R$
$$
u \mapsto L_n(u + \betaLS) - L_n(\betaLS),
$$	
which can, using the normal equations of the LS estimator, be
rewritten to
$$
u'X'Xu + 2 \sum_{j=1}^p \lambda_j \frac{|u_j + \betaLSj| - |\betaLSj|}{|\betaLSj|}.
$$
Note that $G_n$ is minimized at $\betaAL - \betaLS$ and that, since
all directional derivatives have to be non-negative at the minimizer
of a convex function, after some basic calculations we get
\begin{align} \label{eq:dir_dev}
\begin{split}
\D_{e_j} G_n(\betaAL-\betaLS) & \; = \; 2(X'X(\betaAL-\betaLS))_j + 2 \frac{\lambda_j}{|\betaLSj|}
\left(\ind_{\{\betaALj \geq 0\}} - \ind_{\{\betaALj < 0\}}\right) \geq 0 \\
\D_{-e_j} G_n(\betaAL-\betaLS) & \; = \; -2(X'X(\betaAL-\betaLS))_j + 2 \frac{\lambda_j}{|\betaLSj|}
\left(\ind_{\{\betaALj \leq 0\}} - \ind_{\{\betaALj > 0\}}\right) \geq 0 
\end{split}
\end{align}
for all $j = 1,\dots,p$. When $\betaALj = 0$, this implies that
$$
|(X'X(\betaAL-\betaLS))_j| \leq \frac{\lambda_j}{|\betaLSj|}
$$
and therefore
\begin{equation} \label{eq:AL_LS1}
|(\betaAL-\betaLS)_j(X'X(\betaAL-\betaLS))_j| \leq \lambda_j
\end{equation}
holds. When $\betaALj \neq 0$, the equations in \eqref{eq:dir_dev} imply
\begin{equation} \label{eq:AL_LS2}
(X'X(\betaAL-\betaLS))_j = -\lambda_j\frac{\sgn(\betaALj)}{|\betaLSj|}.
\end{equation}
If $|\betaALj-\betaLSj| \leq |\betaLSj|$, clearly, \eqref{eq:AL_LS1}
also holds. If $|\betaALj-\betaLSj| > |\betaLSj|$, we have
$\sgn(\betaALj - \betaLSj) = \sgn(\betaALj) \neq 0$ yielding
$$
(\betaAL-\betaLS)_j(X'X(\betaAL-\betaLS))_j = 
-\lambda_j\frac{|\betaALj-\betaLSj|}{|\betaLSj|} \leq 0.
$$
In any case, $\lambda_j = 0$ implies $(X'X(\betaAL - \betaLS))_j = 0$, which
completes the proof.
\end{proof}

\smallskip

\begin{proof}[Proof of Corollary~\ref{cor:AL_LS}]
By Lemma~\ref{lem:AL_LS}, we have 
$$
0 \leq \sqrt{n}(\betaAL - \betaLS)'\frac{X'X}{n}\sqrt{n}(\betaAL -
\betaLS) \leq \sum_{j=1}^p \lambda_j \leq p \lambda^* \to 0.
$$
Since $X'X/n \to C$ with $C$ being positive definite, the claim
follows.
\end{proof}

\subsection{Proofs for Section~\ref{sec:consistency}}

\begin{proof}[Proof of Proposition~\ref{prop:rate_point}]
Consider the function $H_{n,\beta}: \R^p \to \R$ defined by $H_{n,\beta}(u) =
a_n^2(L_n(u/a_n + \beta) - L_n(\beta))/n$ which can be written as
$$
H_{n,\beta}(u) = u'\frac{X'X}{n}u - \frac{2a_n}{n}u'X'\eps + 2\sum_{j=1}^p \lambda_j
\frac{a_n^2}{n|\betaLSj|}\left(|\frac{u_j}{a_n} + \beta_j| - |\beta_j| \right).
$$
$H_{n,\beta}$ is minimized at $a_n(\betaAL - \beta)$ and, since
$H_{n,\beta}(0) = 0$, we have $H_{n,\beta}(a_n(\betaAL - \beta)) \leq
0$, which implies that
$$
a_n(\betaAL-\beta)'\frac{X'X}{n}a_n(\betaAL-\beta) \leq
\frac{a_n}{\sqrt{n}} a_n(\betaAL-\beta)'\frac{2}{\sqrt{n}}X'\eps  +
2 \sum_{j \in \A} \frac{1}{|\betaLSj|}\frac{a_n\lambda_j}{n} |a_n(\betaAL-\beta)_j|,
$$ 
where in the latter sum we have dropped the non-positive terms for $j
\notin \A$ and have used the fact that $|\beta_j| - |u_j/a_n +
\beta_j| \leq |u_j/a_n|$ on the terms for $j \in \A$. Now note that
both $a_n/\sqrt{n}$ and $a_n \lambda_j/n$ are bounded by 1 and that
the sequences $X'\eps/\sqrt{n}$ and $1/\betaLSj$ for $j \in \A$ are
tight, so that we can bound the right-hand side of the above
inequality by a term that is stochastically bounded times
$\|a_n(\betaAL - \beta)\|$. Moreover, since $X'X/n$ converges to $C$
and all matrices are positive definite, we can bound the left-hand
side of the above inequality from below by a positive constant times
$\|a_n(\betaAL - \beta)\|^2$, so that we can arrive at
$$
\|a_n(\betaAL - \beta)\|^2 \leq O_p(1)\, \|a_n(\betaAL - \beta)\|
$$
which proves the claim.
\end{proof}

\smallskip

\begin{proof}[Proof of Proposition~\ref{prop:rate_unif}]
Let $L > 0$ denote the infimum of all eigenvalues of $X'X/n$ and $C$
taken over $n$ and note that $b_n^2\lambda^*/n \leq 1$. By
Lemma~\ref{lem:AL_LS} we have
$$
b_n^2 \|\betaAL-\betaLS\|^2 \leq 
\frac{b_n^2}{L} (\betaAL-\betaLS)'\frac{X'X}{n}(\betaAL-\betaLS) \leq p
\frac{b_n^2}{L} \frac{\lambda^*}{n} \leq \frac{p}{L}.
$$
For any $M \geq 2 \sqrt{\frac{p}{L}}$ we therefore have
\begin{align*}
\P_\beta(b_n\|\betaAL - \beta\| > M) & \; \leq \; \P_\beta(b_n\|\betaAL - \betaLS\| > M/2) + 
\P(b_n\|\betaLS - \beta\| > M/2) \\ & \; = \; \P(b_n\|\betaLS - \beta\| > M/2).
\end{align*}
The claim now follows from the uniform $\sqrt{n}$-consistency of the
LS estimator.
\end{proof}

\smallskip

\begin{proof}[Proof of Theorem~\ref{thm:consist_param}]
We have \ref{item:cond_lambda} $\implies$ \ref{item:unif_consist} by
Proposition~\ref{prop:rate_unif} and clearly, \ref{item:unif_consist}
$\implies \ref{item:ptw_consist}$ holds. To show
\ref{item:ptw_consist} $\implies$ \ref{item:cond_lambda}, assume that
$\betaAL$ is consistent for $\beta$ and that $\lambda_j/n_k \to c \in
(0,\infty]$ for some $j$ along a subsequence $n_k$. Let $\beta_j \neq
0$. On the event $\betaALj \neq 0$, which by consistency has
asymptotic probability equal to 1, we have
$$
\left|\left(\frac{X'X}{n_k}(\betaAL-\betaLS)\right)_j\right| = \frac{\lambda_j}{n_k|\betaLSj|}
$$
by Equation~\eqref{eq:AL_LS2}. By consistency and the
convergence of $X'X/n$, the left-hand side converges to zero in
probability, whereas the right-hand side converges to $c/|\beta_j| >
0$ in probability along the subsequence $n_k$, yielding a
contradiction. This shows the equivalence of the first three
statements.

Moreover, \ref{item:ptw_consist} $\implies$ \ref{item:correct_ms} since for $j \in \A$
$$
\P_\beta(\betaALj = 0) \leq \P_\beta(|\betaALj - \beta_j| > |\beta_j|/2) \to 0
$$
by consistency in parameter estimation. 

The final implication we show is \ref{item:correct_ms} $\implies$
\ref{item:cond_lambda}. For this, assume that $\lambda^*/n \not\to 0$
so that there exists a subsequence $n_k$ such that $\lambda_j/n_k \to
c > 0$ as $n_k \to \infty$ for some $j$. We first look at the case of
$c = \infty$. Note that $\betaAL$ is stochastically bounded, since
$L_n(\betaAL) \leq L_n(0) = \|y\|^2$ implies
$$
\betaAL'\frac{X'X}{n}\betaAL \leq \betaAL'\frac{X'X}{n}\betaAL + 
2\sum_{j=1}^p\lambda_j \frac{|\betaALj|}{|\betaLSj|} \leq \betaAL'\frac{2}{n}X'y.
$$
As $X'X/n \to C$ and $X'y/n \to X'X\beta$, the quadratic term on the
left-hand side dominates the linear term on the right-hand side which
is only possible if $\betaAL$ is $O_p(1)$. Now note that by
Equation~\ref{eq:AL_LS2}, $\betaALj \neq 0$ implies
$$
\left|\left(\frac{X'X}{n_k}(\betaAL-\betaLS)\right)_j\right| =
\frac{\lambda_j}{n_k}\frac{1}{|\betaLSj|}.
$$
The fact that $X'X/n_k \to C$ and that $\betaAL$ and $\betaLS$ are
stochastically bounded for fixed $\beta$ shows that the left-hand side
of the above display is bounded in probability also. The right-hand
side, however, diverges to $\infty$ regardless of the value of
$\beta_j$. We therefore have $\P_\beta(\betaALj = 0) \to 1$ for all
$\beta_j \in \R$, which is a contradiction to \ref{item:correct_ms}.
If $c < \infty$, we first observe that $X'X/n(\betaAL - \betaLS)$ is
always contained in a compact set by Lemma~\ref{lem:AL_LS} and the
convergence of $X'X/n$ to $C$. This implies that $\|X'X/n(\betaAL -
\betaLS)\|_\infty \leq L < \infty$ for some $L > 0$ and for all
$\beta$. Again, by Equation~\ref{eq:AL_LS2},
$$
\left|\left(\frac{X'X}{n_k}(\betaAL-\betaLS)\right)_j\right| =
\frac{\lambda_j}{n_k}\frac{1}{|\betaLSj|},
$$
whenever $\betaALj \neq 0$. The left-hand side is bounded by $L$
whereas the right-hand side converges to $c/|\beta_j|$ in probability.
We therefore get $\P_\beta(\betaALj = 0) \to 1$ for all $\beta_j \in
\R$ satisfying $|\beta_j| < c/L$, also yielding a contraction to
\ref{item:correct_ms}.
\end{proof}

\smallskip

\begin{proof}[Proof of Theorem~\ref{thm:consist_model}]
Since the condition $\lambda^*/n \to 0$ guards against false negatives
asymptotically by Theorem~\ref{thm:consist_param}, we only need to
show that the estimator detects all zero coefficients with asymptotic
probability equal to one. Assume that $\beta_j = 0$ and that $\betaALj
\neq 0$. The partial derivative of $L_n$ with respect to $b_j \neq 0$
is given by
$$
\frac{\partial L_n}{\partial b_j} = 2(X'Xb)_j - 2(X'y)_j + 2 \frac{\lambda_j}{|\betaLSj|}\sgn(b_j)
= 2(X'X(b - \beta))_j - 2(X'\eps)_j + 2 \frac{\lambda_j}{|\betaLSj|}\sgn(b_j),
$$
which yields
$$
\left|\left(\frac{X'X}{n}(a_n(\betaAL - \beta))\right)_j 
- \frac{a_n}{\sqrt{n}}\frac{1}{\sqrt{n}}(X'\eps)_j \right| = 
\frac{\lambda_j}{\sqrt{n}|\betaLSj|}\frac{a_n}{\sqrt{n}}.
$$
Since $\betaAL$ is $a_n$-consistent for $\beta$, $X'X/n$ converges,
$a_n/n^{1/2} \leq 1$ and $X'\eps/\sqrt{n}$ is tight, the left-hand
side of the above display is stochastically bounded. The behavior of
the right-hand side is governed by $\lambda_j a_n/\sqrt{n}$ as
$\sqrt{n}\betaLSj$ is also stochastically bounded for $\beta_j = 0$. If
$a_n/\sqrt{n}$ does not converge to zero, then the right-hand side
diverges because $\lambda_j$ does. If $a_n/\sqrt{n} \to 0$, we have
$a_n = n/\lambda^*$ eventually, so that $\lambda_j a_n/\sqrt{n} =
\sqrt{n}\lambda_j/\lambda^*$ which also diverges by assumption.
\end{proof}

\subsection{Proofs for Section~\ref{sec:asymp_dist}}

\begin{lemma} \label{lem:Anj}
Assume that $\lambda^*/n \to 0$ and $\lambda^* \to \infty$. Moreover,
suppose that $\psi_{n,j} = \sqrt{\lambda^*}/\lambda_j \to \psi_j \in
[0,\infty]$ and $\phi_{n,j} =
\sqrt{n}\beta_{n,j}\sqrt{\lambda^*}/\lambda_j \to \phi_j \in \Rquer$.
Then for any $u_j \in \R$, the term
$$
\Anj(u_j) = \frac{\lambda_j}{\sqrt{n\lambda^*}}\frac{1}{|\betaLSj|}
\left(|u_j + \sqrt{\frac{n}{\lambda^*}}\beta_{n,j}| -
|\sqrt{\frac{n}{\lambda^*}}\beta_{n,j}|\right)
$$
satisfies $\Anj(u_j) \dto \Aj(u_j)$ where
$$
\Aj(u_j) = \begin{cases}
0 & u_j=0 \text{ or } |\phi_j| = \infty \text{ or } \psi_j = \infty \\ 
\infty & u_j \neq 0 \text{ and } \phi_j = \psi_j = 0 \\ 
2 \frac{|u_j + \lambda^0_j\phi_j| - |\lambda^0_j\phi_j|}{|\psi_jZ_j + \phi_j|} & \text{else}
\end{cases}
$$
with $Z \sim N(0,\sigma^2 C^{-1})$. Moreover,
$$
\sum_{j=1}^p \Anj(u_j) \dto \sum_{j=1}^p \Aj(u_j)
$$
for all $u \in \R^p$.
\end{lemma}

\begin{proof}[Proof of Lemma~\ref{lem:Anj}]
Note that if $u_j = 0$, the term $\Anj$ is clearly equal to 0, so that
we assume $u_j \neq 0$ in the following. Define $\zeta_{n,j} =
\sqrt{n/\lambda^*}\beta_{n,j} \to \zeta_j \in \Rquer$ and notice that
$|\zeta_j| \leq |\phi_j|$, as well as $\zeta_j = \lambda^0_j \phi_j$
when $\lambda^0_j > 0$ or $|\phi_j| < \infty$. Moreover, let $Z_n =
\sqrt{n}(\betaLS - \beta_n)$ which satisfies $Z_n \dto Z$ with $Z \sim
N(0,\sigma^2 C^{-1})$.

We now look at the case where $|\phi_j| = \infty$. The term
$|\Anj(u_j)|$ is bounded by
$$
\frac{\lambda_j}{\lambda^*}\frac{|u_j|}{|Z_{n,j}/\sqrt{\lambda^*} + \zeta_{n,j}|},
$$
where $Z_{n,j}/\sqrt{\lambda^*}$ is $o_p(1)$. If $|\zeta_j| = \infty$
also, the above expression tends to zero in probability. If $0 <
|\zeta_j| < \infty$, the same expression converges to
$\lambda^0_j|u_j|/|\zeta_j|$ in probability. But in this case, we
necessarily have $\lambda^0_j = 0$, so that the limit also equals
zero. If $\zeta_j = 0$, rewrite the above bound to 
$$
\frac{|u_j|}{|\psi_{n,j}Z_{n,j} + \phi_{n,j}|}
$$
which clearly converges to zero in probability when $\psi_j < \infty$.
If $\psi_j = \infty$, note that the above display converges to zero in
probability if and only if for any $\delta > 0$, the expression
\begin{align*}
\P\left(\frac{1}{|\psi_{n,j}Z_{n,j} + \phi_{n,j}|} \geq \delta\right) =
\P\left(|\psi_{n,j}Z_{n,j} + \phi_{n,j}| \leq \frac{1}{\delta}\right) =
\P\left(\frac{-1/\delta - \phi_{n,j}}{\psi_{n,}} \leq Z_n \leq 
\frac{1/\delta - \phi_{n,j}}{\psi_{n,}}\right)
\end{align*}
converges to zero, which it does by Polya's Theorem.

We next turn to the case where $\psi_j = \infty$. If $|\phi_j| =
\infty$ also, the limit equals zero by the above. If $|\phi_j| <
\infty$, since $|\Anj(u_j)|$ is bounded by
$$
\frac{|u_j|}{|\psi_{n,j}Z_{n,j} + \phi_{n,j}|},
$$
it will converge to zero in probability.

Let us now consider the case where $\phi_j = \psi_j = 0$. We write
$\Anj(u_j)$ as
$$
\frac{|u_j + \zeta_{n,j}| - |\zeta_{n,j}|}{|\psi_{n,j}Z_{n,j} + \phi_{n,j}|},
$$
which clearly diverges as $u_j \neq 0$, $|\zeta_{n,j}| \leq
|\phi_{n,j}| \to 0$ and  the denominator tends to 0 in probability.

For the remaining cases where $u_j \neq 0$, $|\phi_j|,\psi_j < \infty$
and $\max(|\phi_j|,\psi_j) > 0$ note that $\Anj(u_j)$ can also be
written as
$$
\frac{|u_j + \zeta_{n,j}| - |\zeta_{n,j}|}{|\psi_{n,j}Z_{n,j} +
\phi_{n,j}|}
$$
and $\zeta_{n,j} \to \zeta_j = \lambda^0_j \phi_j$.

The joint distributional convergence of $\sum_j \Anj(u_j)$ to
$\sum_j \Aj(u_j)$ follows trivially.
\end{proof}

\smallskip

\begin{proof}[Proof of Theorem~\ref{thm:asymp_dist}]
Define $V_{n,\beta_n}(u) =
\frac{1}{\lambda^*}\left(L_n(\sqrt{\lambda^*/n}u + \beta_n) -
L_n(\beta_n)\right)$ and notice that $V_{n,\beta_n}$ is minimized at
$\sqrt{n/\lambda^*}(\betaAL - \beta_n)$. The function $V_{n,\beta_n}$
can be shown to equal
$$
V_{n,\beta_n}(u) = u'\frac{X'X}{n}u - \frac{2}{\sqrt{n\lambda^*}}u'X'\eps +
2 \sum_{j=1}^p \Anj(u_j),
$$
where $\Anj(u_j)$ is defined in Lemma~\ref{lem:Anj}. Since $X'X/n \to
C$, $X'\eps/\sqrt{n}$ is stochastically bounded and $\lambda^* \to
\infty$, invoking Lemma~\ref{lem:Anj} shows that $V_{n,\beta_n}(u)$
converges in distribution to $V_\phi(u)$. We now wish to deduce the
same for the corresponding minimizers $m_n$ and $m$. As explained in
Section~\ref{sec:asymp_dist}, the limiting function $V_\phi$ is not
finite on an open subset of $\R^p$ and we cannot invoke the usual
theorems employed in such a context. Instead, we define a new sequence
of functions whose minimizers behave similarly but whose limiting
function remains finite. To this end, we let $I = \{j :
\max(|\phi_j|,\psi_j) > 0\}$ and assume without loss of generality
that $I = \{1,\dots,\tilde p\}$ with $\tilde p \leq p$ to ease
notation with indices. Now consider $\bar V_{n,\beta_n}: \R^p \to \R$
defined by
$$
\bar V_{n,\beta_n}(u) = u'\frac{X'X}{n}u - \frac{2}{\sqrt{n\lambda^*}}u'X'\eps 
+ 2 \sum_{j \in I} \Anj(u_j)
$$
and let $\tilde V_{n,\beta_n}, \tilde V_\phi: \R^{\tilde p} \to \R$
with
$$ 
\tilde V_{n,\beta_n}(\tilde u) = \bar V_{n,\beta_n}\left(\begin{smallmatrix} \tilde u \\ m_{n,I^c}
\end{smallmatrix}\right) \;\; \text{ and } \;\;
\tilde V_\phi(\tilde u) = V_\phi\left(\begin{smallmatrix} \tilde u \\ 0 \end{smallmatrix}\right).
$$
We first show that $m_{n,I^c} \pto 0$. Note that $V_{n,\beta_n}(m_n)
\leq V_{n,\beta_n}(0) = 0$ implies that
$$
m_n'\frac{X'X}{n}m_n - \frac{2}{\sqrt{n\lambda^*}}m_n'X'\eps + 
2\sum_{j \in I} \Anj(m_{n,j}) \leq -2\sum_{j \notin I} \Anj(m_{n,j}).
$$
The sequence $m_n$ is stochastically bounded by
Proposition~\ref{prop:rate_unif}. But then so is the left-hand side of
the above inequality by Lemma~\ref{lem:Anj}. The right-hand side,
however, tends to $-\infty$ whenever $m_{n,I^c}$ does not tend to zero
in probability, yielding a contradiction.

Since $m_{n,I^c} \pto 0$, it is straightforward to see that $\tilde
V_{n,\beta_n}(\tilde u) \dto \tilde V_\phi(\tilde u)$ for each $\tilde
u \in \R^{\tilde p}$ by Lemma~\ref{lem:Anj}. Inspired by the Convexity
Lemma of \cite{Pollard91}, it can be shown that the functions also
converge uniformly on compact sets of $\R^{\tilde p}$. Since $\tilde
V_{n,\beta_n}$ and $\tilde V_\phi$ are convex and finite, this means
that $\tilde V_{n,\beta_n}$ epiconverges to $\tilde V_\phi$
\citep[c.f.][p.\ 2]{Geyer96TR}. Through Theorem~3.2 in that same
reference, we may deduce that
$$
\argmin_{\tilde u \in \R^{\tilde p}} \tilde V_{n,\beta_n}(\tilde u) \dto 
\argmin_{\tilde u \in \R^{\tilde p}} \tilde V_\phi(\tilde u).
$$
To piece together the missing parts for the minimizers $m_n$ and $m$
of $V_{n,\beta_n}(u)$ and $V_\phi(u)$, respectively, we do the
following. First note that $m_{I^c} = 0$ since otherwise $V_\phi$ is
infinite, so that we have
$$
m_{n,I^c} \pto m_{I^c}.
$$
To finish, observe that
$$
m_{n,I} = \argmin_{\tilde u \in \R^{\tilde p}} \tilde V_{n,\beta_n}(\tilde u)  \dto 
\argmin_{\tilde u \in \R^{\tilde p}} \tilde V_\phi(\tilde u) = m_I.
$$
\end{proof}

\smallskip

\begin{proposition} \label{prop:Vphi_KKT}
The point $m \in \R^p$ is a minimizer of $V_\phi$ if and only if
$$
\begin{cases}
m_j = 0 & \phi_j = \psi_j = 0 \\
(Cm)_j = 0 & |\phi_j| = \infty \text{ or } \psi_j = \infty \\
(Cm)_j = -\frac{\sgn(m_j + \lambda^0_j\phi_j)}{|\psi_j Z_j + \phi_j|}  & 
0 < \max(|\phi_j|,\psi_j) < \infty \text { and } m_j \neq -\lambda^0_j\phi_j \\
|(Cm)_j| \leq \frac{1}{|\psi_j Z_j + \phi_j|} & 
0 < \max(|\phi_j|,\psi_j) < \infty \text{ and } m_j =
-\lambda^0_j\phi_j. \\
\end{cases}
$$
\end{proposition}

\begin{proof}[Proof of Proposition~\ref{prop:Vphi_KKT}]
Clearly, $m_j = 0$ if $\phi_j = \psi_j = 0$ as otherwise $V_\phi$ is
infinite. The other conditions immediately follow by noting that $m$
is a minimizer of the convex function $V_\phi$ if and only if $0$ is a
subgradient of $V_\phi$ at $m$.
\end{proof}

\smallskip

\begin{proof}[Proof of Proposition~\ref{prop:min_set}] 
``$\subseteq$'': We first show that the union of minimizers is
contained in the set $\M$. For this, let $m = \argmin_u V_\phi(u)$ for
some $\phi \in \Rquer^p$. We distinguish three cases.

\smallskip

Firstly, if $\phi_j = \psi_j = 0$, we have $m_j = 0$ which immediately
implies $m_j (Cm)_j = 0 \leq \lambda^0_j$.

\smallskip 

If secondly $|\phi_j| = \infty$ or $\psi_j = \infty$,
Proposition~\ref{prop:Vphi_KKT} implies that $(Cm)_j = 0$ which also
yields $m_j (Cm)_j = 0 \leq \lambda^0_j$.

\smallskip

Thirdly, if $0 < \max(|\phi_j|,\psi_j) < \infty$, we consider two
subcases. When $\psi_j > 0$, $\lambda^0_j = 0$ necessarily holds.
Here, if $m_j = 0$, we immediately have $m_j (Cm)_j = 0 =
\lambda^0_j$. Otherwise, $m_j \neq 0$ implies
$$
m_j(Cm)_j = -\frac{|m_j|}{|\psi_j Z_j + \phi_j|} < 0 = \lambda^0_j
$$
by Proposition~\ref{prop:Vphi_KKT}. The other subcase of $\psi_j = 0$
can be treated as follows. If $m_j = -\lambda^0_j\phi_j$,
Proposition~\ref{prop:Vphi_KKT} yields
$$
|(Cm)_j| \leq \frac{1}{|\phi_j|}
$$
so that
$$
m_j (Cm)_j \leq |m_j (Cm)_j| \leq \frac{|\lambda^0_j\phi_j|}{|\phi_j|} = \lambda^0_j.
$$
If $m_j \neq -\lambda^0_j\phi_j$, the same proposition gives
$$
(Cm)_j = -\frac{\sgn(m_j + \lambda^0_j\phi_j)}{|\phi_j|}.
$$
If $|m_j| > |\lambda^0_j\phi_j|$, we have $\sgn(m_j) = \sgn(m_j +
\lambda^0_j\phi_j)$ and
$$
m_j(Cm)_j = -\frac{|m_j|}{|\phi_j|} < 0 \leq \lambda^0_j.
$$
Finally, if $|m_j| \leq |\lambda^0_j\phi_j|$, similarly to above we
get
$$
m_j (Cm)_j \leq |m_j (Cm)_j| = \frac{|m_j|}{|\phi_j|} \leq 
\frac{|\lambda^0_j\phi_j|}{|\phi_j|} = \lambda^0_j.
$$

\bigskip

\noindent ``$\supseteq$'': We now need to show that for any $m \in
\M$, we can construct a $\phi \in \Rquer^p$, such that $m  = \argmin_u
V_\phi(u)$. To this end, we define
\begin{equation} \label{eq:phi}
\phi_j = \begin{cases}
\infty & (Cm)_j = 0 \\
-\frac{m_j}{\lambda^0_j} & 
(Cm)_j \neq 0 \text{ and } \lambda^0_j > 0 \text{ and } |m_j(Cm)_j| \leq \lambda^0_j \\
\frac{1}{(Cm)_j} - \psi_j Z_j & \text{else} 
\end{cases}
\end{equation}
and show that $m$ is a minimizer of the resulting function $V_\phi$.
First note that since $m \in \M$, $\psi_j = \infty$ immediately
implies $(Cm)_j = 0$, satisfying the second condition of
Proposition~\ref{prop:Vphi_KKT}. We therefore assume that $\psi_j <
\infty$ in the following and go through the three definitions in
\eqref{eq:phi}.

\smallskip

If $(Cm)_j = 0$ then the second condition in
Proposition~\ref{prop:Vphi_KKT} is satisfied.

\smallskip 

When $\phi_j = - m_j/\lambda^0_j$ the condition $\lambda^0_j > 0$
implies that $\psi_j = 0$. So when $m_j = 0$, we are in the case where
$\phi_j = \psi_j = 0$ and the first condition in
Proposition~\ref{prop:Vphi_KKT} is fulfilled. If $m_j \neq 0$, we have
$$
|(Cm)_j| \leq \frac{\lambda^0_j}{|m_j|} = \frac{1}{|\phi_j|}
$$
and the fourth condition in Proposition~\ref{prop:Vphi_KKT} is
satisfied.

\smallskip 

Finally, when $\phi_j = 1/(Cm)_j - \psi_j Z_j$ and $\lambda^0_j > 0$,
we again have $\psi_j = 0$ and therefore $\phi_j = 1/(Cm)_j$. In that
case, we also have $|m_j(Cm)_j| > \lambda^0_j$ which, since $m \in
\M$, implies that $m_j(Cm)_j < 0$, so that we have $\sgn((Cm)_j) =
-\sgn(m_j)$. But this also entails $|m_j| > \lambda^0_j/|(Cm)_j| =
|\lambda^0_j\phi_j|$ so that $m_j \neq -\lambda^0_j\phi_j$ as well as
$\sgn(m_j) = \sgn(m_j + \lambda^0_j\phi_j)$. Thus,
$$
(Cm)_j = \sgn((Cm)_j)|(Cm)_j| = -\frac{\sgn(m_j)}{|\phi_j|} = 
-\frac{\sgn(m_j + \lambda^0_j\phi_j)}{|\phi_j|}
$$
and the third condition in Proposition~\ref{prop:Vphi_KKT} holds.
Lastly, if $\lambda^0_j = 0$ here and $m_j = 0$, it is easily seen
that the fourth condition of Proposition~\ref{prop:Vphi_KKT} is
satisfied. If $m_j \neq 0$, we are again in the case where $m_j \neq
-\lambda^0_j\phi_j$. Since $m \in \M$, we get $m_j (Cm)_j \leq
\lambda^0_j = 0$ and $m_j \neq 0$ and $(Cm)_j \neq 0$ implies
$\sgn((Cm)_j) = - \sgn(m_j)$. Therefore, similarly as above,
$$
(Cm)_j  = \sgn((Cm)_j)|(Cm)_j| = -\frac{\sgn(m_j)}{|\psi_j Z_j + \phi_j|}
$$
holds, satisfying the third condition in
Proposition~\ref{prop:Vphi_KKT}.
\end{proof}

\subsection{Proofs for Section~\ref{sec:conf_sets}}

\begin{proof}[Proof of Theorem~\ref{thm:conf_sets}] 
We start by proving the first statement. Let $g_n(\beta) =
P_\beta(\beta \in \betaAL - \sqrt{\frac{\lambda*}{n}}\O)$ and $c_n =
\inf_{\beta \in \R^p} g_n(\beta)$. We have to show that $c_n \to 1$ as
$n \to \infty$.  Since $c_n$ are the infima of $g_n$ we can choose
sequences $(\tilde \beta_{n,k})_{k \in \N} \subseteq \R^p$ such that
$$
|c_n - g_n(\tilde\beta_{n,k})| \leq \frac{1}{k}
$$
for all $n, k \in \N$. Let $\beta_n = \tilde\beta_{n,n}$ and note that
$|c_n - g_n(\beta_n)| = o(1)$ as $n \to \infty$, so that we can look
at the limiting behavior of $g_n(\beta_n)$ instead. For
$\sqrt{n}\beta_n\frac{\sqrt{\lambda*}}{\lambda_j} \to \phi_j \in
\Rquer$, by Theorem~\ref{thm:asymp_dist}, the Portmanteau Theorem and
Proposition~\ref{prop:min_set} we immediately get
\begin{align*}
1 \geq \limsup_n g_n(\beta_n) & \geq \liminf_n g_n(\beta_n) = \liminf_n
P_{\beta_n}(\sqrt{\frac{n}{\lambda^*}}(\betaAL - \beta_n) \in \O) \\ 
& \geq P_\phi(\argmin_u V_\phi(u) \in \O) \geq P_\phi(\argmin_u V_\phi(u) \in \M) = 1,
\end{align*}
proving that $\lim_n c_n = \lim_n g_n(\beta_n) = 1$.

\smallskip

To show the second statement, we define a specific point $m$ on the
boundary of $\M$, as well as $\phi \in \Rquer^p$ such that $m =
\argmin_u V_\phi(u)$ and $\sqrt{n/\lambda^*}(\betaAL - \beta) \pto m$,
implying that the limiting distribution is non-random. Hence,
excluding an open set around that $m$ of $\M$ will give an infimal
coverage probability tending to $0$. Towards this end, let $\S = \{j :
\lambda^0_j > 0\}$ and note that $\S \neq \emptyset$ so that we have
$r = C^{-1}\lambda^0 \neq 0$. Moreover,
$$
0 < r'Cr = \sum_{j \in \S} \lambda^0_j r_j 
$$
implies that there is at least one positive component $r_j$ with $j
\in \S$. Now define $r_0 = \max_{j \in S} r_j > 0$, let $m =
r_0^{-1/2} r$ and note that this $m$ satisfies $m \in
\M\setminus\M_d$, since $Cm = r_0^{-1/2}\lambda^0$ and
$$
m_j(Cm)_j = \lambda^0_j \, \frac{r_j}{r_0},
$$
implying that $(Cm)_j = 0$ for $j \notin \S$, $m_j (Cm)_j \leq
\lambda^0_j$ for $j \in S$ and $m_j (Cm)_j = \lambda^0_j >
d\lambda^0_j$ for some $j \in S$. Also note that $\psi_j = \infty$
implies $j \notin \S$. Now let $\phi \in \Rquer^p$ with
$$
\phi_j  = \begin{cases}
\infty & (Cm)_j = 0 \\
-\frac{m_j}{\lambda^0_j} & (Cm)_j \neq 0 \text{ and } |m_j(Cm)_j| \leq \lambda^0_j \\
\frac{1}{(Cm)_j} & \text{else}.
\end{cases}
$$
According to \eqref{eq:phi} in the proof of
Proposition~\ref{prop:min_set}, $m$ then is the unique minimizer of
the corresponding function $V_\phi$. This can be seen by noting that
$(Cm)_j = 0$ if and only if $\lambda^0_j = 0$, as well as $\psi_j > 0$
implying that $\lambda^0_j = 0$. It is crucial to observe that the
function $V_\phi$ is non-random in this case and that $\M_d$ is
closed. Now take any sequence $(\beta_n)_{n \in \N} \subseteq \R^p$
converging to $\phi$ and let $f_n(\beta) = P_{\beta}(\beta \in \betaAL
- \sqrt{\frac{\lambda^*}{n}}\M_d)$. By Theorem~\ref{thm:asymp_dist}
and the Portmanteau Theorem we have
\begin{align*}
0 & \leq \liminf_n \inf_{\beta \in \R^p} f_n(\beta) 
\leq  \limsup_n \inf_{\beta \in \R^p} f_n(\beta) 
\leq \limsup_n P_{\beta_n}(\sqrt{\frac{n}{\lambda^*}}(\betaAL - \beta_n) \in \M_d) \\
& \leq P_\phi(\argmin_u V_\phi(u) \in \M_d) = \ind_{\{m \in \M_d\}} = 0.
\end{align*}

\end{proof}

The following lemma is the basis to prove
Theorem~\ref{thm:conf_bounds}. For a symmetric matrix $A$, we denote
by $\kappa_A$ the condition number of $A$ with respect to the spectral
norm, i.e., the ratio of the largest by the smallest eigenvalue of $A$
(in absolute value).

\begin{lemma} \label{lem:finIneq}
Let 
$$
c_n = \min_{1 \leq j \leq p} \lambda_j
\left(\dfrac{d_n-1}{2 \sqrt{\kappa_{X'X} l_n \lambda^*}}\right)
$$
with $l_n = \sum_{j=1}^p \lambda_j/\lambda^*$. If $d_n \geq 1$ we have
$$
\inf_{\beta \in \R^p} 
\P_\beta\left(\beta \in \betaAL - \sqrt{\frac{\lambda^*}{n}}\Mhat_{d_n} \right) \geq 
\P\left(\sqrt{\eps'X(X'X)^{-1}X'\eps} \leq 
c_n \left(1 -\frac{c_n}{2 \sqrt{l_n \lambda^*}} \right)\right).
$$
\end{lemma}

\begin{proof}
Let $a = c_n (1 - c_n/(2 \sqrt{l_n \lambda^*}))$. The above statement
is trivial when $a < 0$. Note that by Lemma~\ref{lem:AL_LS}, $\betaAL
- \betaLS$ is an element of $\sqrt{\lambda^*/n}\,\Mhat_1$. If $a = 0$,
the event on the right-hand side implies $X'\eps = 0$ and therefore
$\betaLS - \beta = (X'X)^{-1}X'\eps = 0$. But then we get $\betaAL -
\betaLS = \betaAL - \beta \in \sqrt{\lambda^*/n}\,\Mhat_1$, which
implies the claim since $\Mhat_1 \subseteq \Mhat_{d_n}$.

\smallskip

We now prove the statement for $a > 0$. If we can show that
whenever $z'X'Xz \leq a^2$ and $m \in \sqrt{\lambda^*/n}\,\Mhat_1$, we
get $z + m \in \sqrt{\lambda^*/n}\,\Mhat_{d_n}$, then the following
holds
\begin{align*}
\P&\left(\eps'X(X'X)^{-1}X'\eps \leq a^2\right) = 
\P_\beta\left((\betaLS - \beta)'X'X(\betaLS - \beta) \leq a^2, 
\betaAL - \betaLS \in \sqrt{\lambda^*/n}\,\Mhat_1\right)  \\
& \leq \P_\beta\left(\betaLS - \beta + \betaAL - \betaLS \in \sqrt{\lambda^*/n}\,\Mhat_{d_n}\right) 
= \P_\beta\left(\beta \in \betaAL - \sqrt{\lambda^*/n}\,\Mhat_{d_n} \right)
\end{align*}
for all $\beta \in \R^p$, which is what we have to prove. It only
remains to show that $z + m \in \sqrt{\lambda^*/n}\,\Mhat_{d_n}$
whenever $z'X'Xz \leq a^2$ and $m \in \sqrt{\lambda^*/n}\,\Mhat_1$. To
do so, we show that $(z+m)_j (X'X (z+m))_j \leq \lambda_j d_n$ for all
$j$. As $a > 0$ implies $\lambda_j > 0$ for all $j$, this suffices to
conclude $z + m \in \sqrt{\lambda^*/n}\,\Mhat_{d_n}$. Clearly, $m \in
\sqrt{\lambda^*/n}\,\Mhat_1$ implies $m'X'Xm \leq \sum_{j=1}^p
\lambda_j$. We also have
$$
L_n \|m\|_\infty^2 \leq L_n \|m\|_2^2 \leq m'X'Xm
$$
and
$$
\|X'Xm\|_\infty^2 \leq \|X'Xm\|_2^2 \leq U_n m'X'Xm,
$$
where $L_n$ and $U_n$ are the smallest and largest eigenvalue of
$X'X$, respectively. With the same argument, we get $\|z\|_\infty^2
\leq a^2/L_n$ and $\|X'Xz\|_\infty^2 \leq a^2 U_n$. Equipped with
these inequalities, we conclude for every $j$ that
\begin{align*}
(z + m)_j (X'X(z + m))_j = & \; z_j (X'Xz)_j + m_j (X'Xm)_j + z_j (X'Xm)_j + m_j (X'Xz)_j \\
\leq & \; a^2 \sqrt{U_n/L_n} + \lambda_j + 2 a \sqrt{(U_n/L_n)l_n\lambda^*} \\
\leq & \; \lambda_j +2c_n(1-c_n/(2\sqrt{l_n\lambda^*}))\sqrt{\kappa_{X'X}l_n\lambda^*} 
+ a^2 \sqrt{\kappa_{X'X}} \\ 
\leq & \; \lambda_j + 2 c_n \sqrt{\kappa_{X'X}l_n\lambda^*} - c_n^2 \sqrt{\kappa_{X'X}} + a^2 
\sqrt{\kappa_{X'X}} \\ 
\leq & \; \lambda_jd_n + (a^2 - c_n^2)\sqrt{\kappa_{X'X}} \leq \lambda_j d_n,
\end{align*}
which completes the proof.
\end{proof}

\begin{remark*}
Lemma~\ref{lem:finIneq} bases on a purely algebraic argument and is
still valid if $X$ and $\lambda$ are stochastic (possibly depending on
$\eps$ and each other) and $\eps$ follows an arbitrary distribution.
The only condition needed is the regularity of $X'X$ with probability
$1$.
\end{remark*}

\begin{proof}[Proof of Theorem~\ref{thm:conf_bounds}]
We start by proving the second statement. Note that if $\nu > 0$, we
have $\delta_n > 1$ eventually, allowing to apply
Lemma~\ref{lem:finIneq}. We have that
$$
c_n = \min_{1 \leq j \leq p} \lambda_j
\left(\dfrac{d_n - 1}{2\sqrt{l_n \kappa_{X'X} \lambda^*}} \right) = 
\min_{1 \leq j \leq p} \left( \dfrac{\sqrt{\lambda^*} (d_n-1)}{2 \sqrt{l_n \kappa_{X'X} }} 
\frac{\lambda_j}{\lambda^*}\right)
\underset{n \to \infty}{\longrightarrow} \underset{1 \leq j \leq p}{\min}~
\frac{\nu \lambda^0_j}{2\sqrt{l_0 \kappa_C}}. 
$$
Moreover, $\eps'X(X'X)^{-1}X'\eps/\sigma^2$ converges to a chi-squared
random variable with $p$ degrees of freedom. The second claim then
follows by Lemma~\ref{lem:finIneq} and Polya's Theorem.

The main idea to show the first claim is the following. We pick a
sequence $\beta_n$ close to the boundary of, but outside the set
$\sqrt{\lambda^*/n}\,\Mhat_{d_n}$. As
$\sqrt{\lambda^*/n}\,\Mhat_{d_n}$ converges to
$\sqrt{\lambda^*/n}\,\Mhat_1$, we expect the LS estimator to lie in
the set $\sqrt{\lambda^*/n}\,\Mhat_1$ with a positive probability.
(This is actually the fact if and only if $\nu \in \R$, because then
the gap between $\sqrt{\lambda^*/n}\,\Mhat_1$ and
$\sqrt{\lambda^*/n}\,\Mhat_{d_n}$ is of order $n^{-1/2}$.) However,
$\betaLS \in \sqrt{\lambda^*/n}\,\Mhat_1$ guarantees $\betaAL = 0$. In
that case, $\betaAL - \beta_n = -\beta_n$ is located outside of
$\Mhat_{d_n}$. Hence, $\P_{\beta_n}(\betaLS \notin
\sqrt{\lambda^*/n}\, \Mhat_1)$ gives an upper bound for the infimal
coverage probability.

For an arbitrary but fixed component $1 \leq s \leq p$, we define
$$
\beta_n = \sqrt{ \dfrac{ \lambda_s d_n + \delta_n}{(X'X)_{ss}} } e_s,
$$
where $\delta_n > 0$ and $\lim_{n \to \infty} \delta_n = 0$.
From $\beta_{n,s} (X'X \beta_n)_s = \lambda_s d_n + \delta_n$ it
follows that $\beta_n \notin \sqrt{\frac{\lambda^*}{n}} \widehat{\M}_{d_n}$. 
Hence,
\begin{align*}
\underset{\beta \in \R^p}{\sup} 
\P_\beta \left( \beta \notin \betaAL + \sqrt{\frac{\lambda^*}{n}} \widehat{\M}_{d_n} \right) \geq
\P_{\beta_n} \left( \beta_n \notin \betaAL + \sqrt{\frac{\lambda^*}{n}} \widehat{\M}_{d_n} \right) \geq
\P_{\beta_n} \left( \betaAL = 0 \right).
\end{align*}

If $| \betaLSj (X'X \betaLS)_j | < \lambda_j$ for all $j$, then
$G_n(u) = L_n(u + \betaLS ) - L_n( \betaLS)$ is minimized at $u = -\betaLS$, which gives $\betaAL = 0$. 
So in order to finish the proof, we only have to show that
$$
\underset{n \to \infty}{\lim} \P_{\beta_n} \left( 
| \betaLSj (X'X \betaLS)_j | < \lambda_j
\text{ for all } j \right) = \Phi
\left( \frac{-\nu \sqrt{\lambda^0_s}}{\sigma\sqrt{3+(C^{-1})_{ss}C_{ss}}} \right).
$$
Since $\beta_{n,j} (X'X \beta_n)_j = 0$ for $j \neq s$, we have $\betaLSj (X'X \betaLS)_j \sim \mathcal{O}_p(\sqrt{\lambda^*})$, 
implying $| \betaLSj (X'X \betaLS)_j | < \lambda_j$ with asymptotic probability $1$ as $\lambda^0_j > 0$.
So it only remains to show that
$$
\underset{n \to \infty}{\lim} \P_{\beta_n} \left( 
| \betaLSs (X'X \betaLS)_s | < \lambda_s \right) =
\Phi \left( \frac{-\nu \sqrt{\lambda^0_s}}{\sigma\sqrt{3+(C^{-1})_{ss}C_{ss}}} \right)
$$
holds true. For this, we use the equality
\begin{align*}
\dfrac{1}{\sqrt{ \lambda_sd_n + \delta_n}} &\left(
\betaLSs (X'X \betaLS)_s - \beta_{n,s} (X'X \beta_n)_s \right) = \\
& \dfrac{1}{\sqrt{ \lambda_sd_n + \delta_n}} \left(
((X'X)^{-1}X'\eps)_s (X'\eps)_s + \beta_{n,s} (X'\eps)_s + 
((X'X)^{-1}X'\eps)_s (X'X\beta_n)_s \right) 
\dto Z,
\end{align*}
where $Z \sim \mathcal{N}(0 , \sigma^2 (C_{ss} (C^{-1})_{ss} + 3))$.
This implies
\begin{align*}
\P_{\beta_n} & \left( 
\betaLSs (X'X \betaLS)_s < \lambda_s \right) = \\
& \P_{\beta_n} \left( \dfrac{1}{\sqrt{ \lambda_sd_n + \delta_n}} \left(
\betaLSs (X'X \betaLS)_s - \beta_{n,s} (X'X \beta_n)_s \right) < 
\dfrac{ \lambda_s(1-d_n)+\delta_n }{\sqrt{ \lambda_sd_n + \delta_n}} \right),
\end{align*}
where the right-hand side inside the probability converges to $-\nu \sqrt{\lambda^0_s}$, even in the case where $\lim_{n \to \infty} d_n \neq 1$.
Since $\lim_{n \to \infty} \P_{\beta_n} \left( 
\betaLSs (X'X \betaLS)_s \leq -\lambda_s \right) = 0$ follows by $\betaLSs (X'X \betaLS)_s / \lambda_s \pto \lim_{n \to \infty} d_n \geq 0$, 
the proof is complete.
\end{proof}

\bibliographystyle{biometrika} \bibliography{journalsFULL,stat,economet}

\end{document}